\documentclass[10.5pt]{amsart}

\usepackage{amsmath}
\usepackage{amssymb}
\usepackage{mathrsfs}
\usepackage[all]{xy}
\usepackage{mathrsfs}

\newtheorem{theorem}{Theorem}[section]
\newtheorem{claim}[theorem]{Claim}

\newtheorem{lemma}[theorem]{Lemma}

\newtheorem{proposition}[theorem]{Proposition}
\newtheorem{corollary}[theorem]{Corollary}

\theoremstyle{definition}
\newtheorem{definition}[theorem]{Definition}

\newtheorem{question}[theorem]{Question}

\theoremstyle{remark}
\newtheorem{remark}[theorem]{Remark}

\newtheorem{fact}[theorem]{Fact}

\def\l{{\langle}}
\def\r{{\rangle}}

\newcount\skewfactor
\def\mathunderaccent#1#2 {\let\theaccent#1\skewfactor#2
\mathpalette\putaccentunder}
\def\putaccentunder#1#2{\oalign{$#1#2$\crcr\hidewidth
\vbox to.2ex{\hbox{$#1\skew\skewfactor\theaccent{}$}\vss}\hidewidth}}

\newcommand{\lusim}[1]{\smash{\underset{\raisebox{1.2pt}[0cm][0cm]{$\sim$}}
{{#1}}}}
\def\smallbox#1{\leavevmode\thinspace\hbox{\vrule\vtop{\vbox
   {\hrule\kern1pt\hbox{\vphantom{\tt/}\thinspace{\tt#1}\thinspace}}
   \kern1pt\hrule}\vrule}\thinspace}

\DeclareMathOperator{\dom}{dom}

\title{Saturation properties of ultrafilters in canonical inner models}
\author{Tom Benhamou}\address[Benhamou]{Department of Mathematics, Statistics, and Computer Science, University of Illinois at Chicago, Chicago, IL 60607, USA}
\email{tomb@uic.edu}

\date{\today}


\begin{document}
\let\labeloriginal\label
\let\reforiginal\ref
\def\ref#1{\reforiginal{#1}}
\def\label#1{\labeloriginal{#1}}

\begin{abstract}
 We improve Galvin's Theorem for ultrafilters which are p-point limits of p-points. This implies that in all the canonical inner models up to a superstrong cardinal, every $\kappa$-complete ultrafilter over a measurable cardinal $\kappa$ satisfies the Galvin property. On the other hand, we prove that supercompact cardinals always carry  non-Galvin $\kappa$-complete ultrafilters. Finally, we prove that $\diamondsuit(\kappa)$ implies the existence of a $\kappa$-complete filter which extends the club filter and fails to satisfy the Galvin property. This answers questions \cite[Question 5.22]{TomMotiII},\cite[Question 3.4]{Non-GalvinFil} and questions ,\cite[Question 4.5]{BenGarShe},\cite[Question 2.26]{bgp}.
 \end{abstract}
\maketitle
\section{introduction}
Fred Galvin discovered the following remarkable saturation property of normal filters \cite[Theorem 3.3]{Galvintheo} (for the proof of the generalized theorem see for example \cite[Prop. 1.4]{GitDensity}): If $\kappa^{<\kappa}=\kappa$, then every normal filter $U$ over $\kappa$ satisfies the \textit{Galvin property}, that is:  $$\forall\l A_i\mid i<\kappa^+\r\in[U]^{\kappa^+}\exists I\in[\kappa^+]^\kappa, \bigcap_{i\in I}A_i\in U.$$
A filter which satisfies the Galvin property is called a \textit{Galvin filter}. This property, as well as some of its variants, have been recently studied in several papers \cite{MR3604115,GartiTilt,ghhm,BenGarShe,NegGalSing,bgp,partOne,Non-GalvinFil}. It has been used by Gitik and the author to extend the study of intermediate models of Magidor-Radin forcing \cite{TomMotiII} and the Tree-Prikry forcing \cite{OnPrikryandCohen}. 

 Galvin's theorem was recently improved for ultrafilter \cite{TomMotiII}, namely, from the same assumption $\kappa^{<\kappa}=\kappa$, we can deduce that  a finite product of $p$-point ultrafilters is a Galvin ultrafilter. In section~\ref{improveGal}, we  improve even more this result and extend the class of filters for which Galvin's property holds to sums of $p$-points:
\vskip 0.3 cm
    \textbf{Theorem \ref{general ppoint}.} 
 Suppose that $W\equiv_{RK}\sum_{U,\alpha} U_\alpha$ where $U$ is a $p$-point ultrafilter over $\kappa$ and for each $\alpha<\kappa$, $U_\alpha$ is a $p$-point ultrafilter over some $\delta_\alpha\leq\kappa$. Then $W$ has the Galvin property. 
\vskip 0.3 cm
Although this seems as a slight improvement to the one which appears in \cite{TomMotiII}, it turns out that it is significant to the analysis of ultrafilters in canonical inner models. 
This result easily generalizes to  filters $W$ which are Rudin-Keisler equivalent to a $p$-point limit of $p$-point limits of $p$-points, and this can be pushed to any finite number of limits of $p$-points.

\subsection{The Galvin property in canonical inner models} In \cite{TomMotiII}, it was noticed that in $L[U]$, every $\kappa$-complete ultrafilter satisfies the Galvin property. The reason is essentially that every $\kappa$-complete ultrafilter over $\kappa$ is Rudin-Keisler isomorphic to a finite product of the unique normal ultrafilter $U$ which exists in $L[U]$, then the theorem about finite product of $p$-points applies. Our aim here is to extend this result to canonical inner models which are compatible with greater large cardinals. More precisely, we will prove that in all known canonical inner models below a superstrong cardinal, every $\kappa$-complete ultrafilter is Rudin-Keisler isomorphic to an $n$-fold limit of $p$-points (see Definition \ref{Def: n-fold}), which then falls under our improvement of Galvin's theorem. We will use  G.~Goldberg's \textit{Ultrapower Axiom}($\mathrm{UA}$) who proved \cite{GoldbergUA} that this axiom follows from weak comparison and is therefore expected to hold in every canonical inner model of set theory. Goldberg proved that $\mathrm{UA}$ has severe impact of the structure of $\sigma$-complete ultrafilters and this is precisely why it fits so well to this paper. 
Our results are applicable to the Mitchell-Steel indexed models $L[\mathbb{E}]$  allowing extenders of superstrong type on the sequence $\mathbb{E}$. Then use F.~Schlutzenberg \cite{Schlutz,Schlutz1} results for those models to prove the following theorem: 

\vskip 0.3 cm
    \textbf{Corollary \ref{GalinCan}.}
 Suppose that there is no measurable limit of superstrong cardinals in the Mitchell-Steel model $L[\mathbb{E}]$. Then every $\kappa$-complete ultrafilter over $\kappa$ satisfies\footnote{ The notation $Gal(\mathcal{F},\kappa,\lambda)$ means that for any collections $\mathcal{A}$ of $\lambda$-many sets in $\mathcal{F}$, there is a sub-collection $\mathcal{B}\subseteq \mathcal{A}$ consisting of $\kappa$-many sets such that $\cap \mathcal{B}\in \mathcal{F}$.}  $Gal(U,\kappa,\kappa^+)$.
\vskip 0.3 cm
At the moment, we do not know whether this result can be further pushed to canonical inner models which are compatible with greater large cardinals, and as we will see next, relates to the open problem of finding a canonical inner model with a supercompact cardinal.
\subsection{Non-Galvin ultrafilters at very large cardinals}
In \cite{TomMotiII}, it was asked if it is consistent to have a $\kappa$-complete ultrafilter which fails to satisfy the Galvin property. This was answered by Garti, Shelah and the author in  \cite{BenGarShe} starting from a supercompact cardinal and improved later by Gitik and the author to a single measurable. In \cite[Question~4.5]{BenGarShe}, the following question appears:
\begin{question}
Is it consistent to have a supercompact cardinal $\kappa$ such that every $\kappa$-complete ultrafilter satisfies the Galvin property?
\end{question}
This question was further investigated in \cite{bgp}, where it was proven that it is consistent that every ground model $\kappa$-complete ultrafilter extends to a Galvin $\kappa$-complete ultrafilter. However, this is not enough as new $\kappa$-complete ultrafilters which do not extend ground model ultrafilters might appear in the generic extension.
The third objective of this paper is to  give a negative answer to this question:

 \vskip 0.3 cm    \textbf{Theorem \ref{SupercompactGalvin}.}
 Suppose that $j:V\rightarrow M$ is an elementary embedding with $crit(j)=\kappa$, $M^{\kappa}\subseteq M$, $j_U:V\rightarrow M_U$ is the derived normal ultrapower and $k:M_U\rightarrow M$ is the factor map. Suppose that $j[\kappa^+],k[j_U(\kappa^+)]\in M$, and that $j_U(\kappa^+)<j(\kappa)$. Then there is a non-Galvin $\kappa$-complete ultrafilter over $\kappa$. 
 \vskip 0.3 cm
 In particular, if $\kappa$ is a $2^\kappa$-supercompact cardinal then there is a non-Galvin $\kappa$-complete ultrafilter. However, this ultrafilter does not necessarily extend the club filter.
 In order to show the existence of such an ultrafilter which does extend the club filter, will combine three techniques, the first is the construction of a special $\kappa$-independent family from a paper by Y.~Hayut \cite{YairIndep}. Secondly, we will use this independent family to construct non-Galvin filters which appears in \cite{Non-GalvinFil} and the finally the techniques from \cite{OnPrikryandCohen} to lift elementary embedding so that a specific $\kappa$-independent family (namely, mutually generic Cohen sets) witnesses the failure of the Galvin property. 

\subsection{Non-Galvin filters extending the club filter}
By applying Galvin's theorem to the club filter $Cub_\kappa$, where $\kappa$ satisfies $\kappa^{<\kappa}=\kappa$, we conclude that $Cub_\kappa$ satisfies the Galvin property, namely, that from every collection $\mathcal{A}$ of $\kappa^+$-many clubs at $\kappa$, one can always extract a subcollection $\mathcal{B}\subseteq \mathcal{A}$ of size $\kappa$ such that $\bigcap\mathcal{B}$ is a club. In order to extend Galvin's theorem, we can start by removing the assumption that $\kappa^{<\kappa}=\kappa$. In this direction, Abraham and Shelah \cite{AbrahamShelah1986} forced a model where $\kappa^{<\kappa}>\kappa$ and $Cub_\kappa$ in a non-Galvin filter, thus proving that the assumption $\kappa^{<\kappa}=\kappa$ is necessary. Another possible extension of the theorem would be to keep the assumption that $\kappa^{<\kappa}=\kappa$, and give up the normality assumption. Theorem \ref{generalgeneral p-point}, is exactly this kind of extension to limits of $p$-point filters. It is impossible to extend the theorem to $q$-points since a non-Galvin $\kappa$-complete ultrafilter which extends the club filter consistently exists \cite{OnPrikryandCohen}. If we do not assume large cardinals, then it does not make sense to ask for non-Galvin $\kappa$-complete ultrafilters. Instead, we can ask for non-Galvin $\kappa$-complete filters, or more interestingly ones which extend the club filter. This was specifically asked in \cite[Question 5.22]{TomMotiII}. In \cite{Non-GalvinFil}, such a filter was constructed from a $\kappa$-independent family, but the filter constructed was not guaranteed to extend the club filter. The construction in this paper provides such a filter assuming that $\diamondsuit$ holds:
\vskip 0.3 cm
\textbf{Theorem \ref{nongalfil}.}
 Suppose that $\kappa$ is a regular cardinal such that $\diamondsuit(\kappa)$ holds, then there is a $\kappa$-complete filter on $\kappa$, extending the club filter, and failing to satisfy the Galvin property.
 \vskip 0.3 cm
 This theorem immediately applies to $L$:
 \vskip 0.3 cm
\textbf{Corollary \ref{Lsituation}.} $L$ is a model of $GCH$ where every regular uncountable cardinal $\kappa$ admits a non-Galvin $\kappa$-complete filter $\mathcal{F}$ which extends the club filter. 
\vskip 0.3 cm

We also apply this construction of a special $\kappa$-independent family to get an ultrafilter on a supercompact cardinal which extend the club filter, this is done in theorem \ref{SupercomClub}.
\subsection*{Notations} If  $f:A\rightarrow B$ is a function we denote by $f[X]=\{f(x)\mid x\in X\}$ and $f^{-1}[Y]=\{a\in A\mid f(a)\in Y\}$.  
Given $\kappa>\omega$ and a $\kappa$-complete ultrafilter $U$, we denote by $M_U$ the transitive collapse of the ultrapower by $U$ and $j_U:V\rightarrow M_U$ the usual ultrapower embedding. If $W,U$ are ultrafilter on $X,Y$ resp.~we define the \textit{Rudin-Keisler} order denoted $U\leq_{RK}W$ if there is a function $\pi:X\rightarrow Y$ such that $U=\{B\subseteq Y\mid \pi^{-1}[B]\in W\}$. We say that $U\equiv_{RK}W$ is $U\leq_{RK}W$ and $W\leq_{RK}U$, or equivalently, if there is $\pi\restriction A$ above is one-to-one for some $A\in W$. If $M$ is any model of $ZFC$, the relativization of the objects to this model are denoted by $()^M$, for example, $(\kappa^+)^M$, $V_\kappa^M$, $j^M_E$, etc. We say that $\kappa$ is a  \textit{measurable cardinal} if it carries a non-principal $\kappa$-complete ultrafilter. We say it is a \textit{$\lambda$-supercompact cardinal} if there is an elementary embedding $j:V\rightarrow M$ with $crit(j)=\kappa$ and $M^{\lambda}\subseteq M$. We say that it is a \textit{superstrong cardinal} if there is an elementary embedding $j:V\rightarrow M$ with $crit(j)=\kappa$ and $V_{j(\kappa)}\subseteq M$.
\section{The Galvin property at limits of ultrafilters}\label{improveGal}
Recall that a $\kappa$-complete  ultrafilter $U$ over $\kappa$ is called a \textit{p-point} if  for every function $f:\kappa\rightarrow\kappa$ which is nonconstant $(mod\ U)$ i.e. for every $\gamma<\kappa$, $f^{-1}[\{\gamma\}]\notin U$, $f$ is almost one-to-one $(mod\ U)$, namely, there exists $X\in U$ such that for each $\gamma<\kappa$, $f^{-1}[\gamma]\cap X$ is bounded. The following proposition is essentially due to Puritz \cite{PURITZ1972215}, who formulated it for $\kappa=\omega$. The generalization to $\kappa>\omega$ appears in Kanamori's  paper \cite{Kanamori1976UltrafiltersOA}. As Puritz's (and Kanamori's) formulation of this proposition involves ``skies" and ``constellations", which we will not define in this paper, let us provide a proof not involving these notions:
\begin{proposition}\label{kanamoriequiv}
Let $U$ be a $\kappa$-complete ultrafilter over $\kappa$. Then $U$ is $p$-point iff for every $\alpha<j_U(\kappa)$ there is $f:\kappa\rightarrow \kappa$ such that $\alpha\leq j_U(f)(\kappa)$.
\end{proposition}
\begin{proof}
Suppose that $U$ is a $p$-point. Let $\kappa\leq [f]_{U}<j_U(\kappa)$ also let $[\pi]_U=\kappa$. Then $\pi,f:\kappa\rightarrow\kappa$ and $f,\pi$ are not constant.
Since $U$ is $p$-point we conclude that there is $X\in U$ such that for every $\gamma<\kappa$, $X\cap f^{-1}[\gamma],X\cap \pi^{-1}[\gamma]$ are bounded in $\kappa$. In particular, $j_U(\pi)([id]_U)=[\pi]_U=\kappa$.  Define $$g(\gamma)=\sup(f[\sup(\pi^{-1}[\gamma+1]\cap X)+1])$$ then $g:\kappa\rightarrow \kappa$. Let us prove that $j_U(g)(\kappa)\geq[f]_U$. Indeed, $$j_U(g)(\kappa)=\sup(j_U(f)[\sup(j_U(\pi)^{-1}[\kappa+1]\cap j_U(X))+1])\geq j_U(f)([id]_U)=[f]_U.$$ In the other direction, let $f$ be not constant $(mod \ U)$.  By assumption, there is a monotone function $g:\kappa\rightarrow \kappa$ such that $j_U(g)(\kappa)\geq [id]_U$. By monotonicity and since $[f]_U\geq\kappa$, $[id]_U\leq j_U(g)(\kappa)\leq j_U(g)([f]_U)$. In particular,  we have that
$$X=\{\nu<\kappa\mid \nu\leq g(f(\nu))\}\in U.$$
Let us prove that $f$ is almost $1-1$ ($mod \ U$). Let $\gamma<\kappa$. 
Define $\Gamma=\sup(g[\gamma])$. Then for every $\alpha\in X\cap f^{-1}[\gamma]$, $\alpha\leq g(f(\alpha))\in g[\gamma]$. Hence $\alpha\leq \Gamma$, and $f^{-1}[\gamma]\cap X\subseteq \Gamma$, as wanted. 
\end{proof}
In this section, we aim to expand the class of ultrafilters which are known to be Galvin. The best result known so far is the one from \cite[Corollary 5.29]{TomMotiII} where the following was proven:
\begin{proposition}\label{galvinGen}
 Suppose that $\kappa^{<\kappa}=\kappa$ and let $F$ be a product of $p$-point filters  over $\kappa$ . Let $\langle X_i\mid i<\kappa^+\rangle$ be a sequence of sets such that for every $i<\kappa^+$, $X_i\in F$, and let $\langle Z_i\mid i<\kappa^+\rangle$ be any sequence of subsets of $\kappa$. Then there is $Y\subseteq \kappa^+$ of cardinality $\kappa$, such that 
 \begin{enumerate}
     \item $\bigcap_{i\in Y}X_i\in F$.
     \item there is $\alpha\notin Y$ such that $[Z_{\alpha}]^{<\omega}\subseteq \bigcup_{i\in Y}[Z_i]^{<\omega}$
 \end{enumerate}
 In particular, $F$ is a Galvin filter.
\end{proposition}
\begin{definition}
 Let $U$ be an ultrafilter over $X$ and $\l U_x\mid x\in X\r$ be a sequence of ultrafilters such that $U_x$ is an ultrafilter over some set $Y_x$. Define the $U-$sum of the $U_x$ as
 the ultrafilter $\sum_{U,\alpha} U_\alpha$ over $\bigcup_{x\in X}\{x\}\times Y_x$ defined  as follows:
 $$B\in \sum_{U,x} U_x\Leftrightarrow \{x\in X\mid (B)_x\in U_x\}\in U$$
 Where $(B)_x=\{y\in Y_x\mid \l x,y\r\in B\}$.
 
 If each $X=\kappa$ and for each $\alpha<\kappa$ $Y_\alpha=\delta_\alpha\leq\kappa$, denote the $U$-limit of the $U_\alpha$'s as the ultrafilter $lim_UU_\alpha$ over $\kappa$ which is defined as follows:
 $X\in lim_UU_\alpha\Leftrightarrow \{\alpha<\kappa\mid X\cap\delta_\alpha\in U_\alpha\}\in U$.
\end{definition}
\begin{fact}\label{fact:RK projection of sum}
 The projection on the right coordinate $\pi_2:\kappa\times\kappa\rightarrow \kappa$ restricted to $\bigcup_{\alpha<\kappa}\{\alpha\}\times\delta_\alpha$ (the space of $\sum_{U,\alpha}U_\alpha$) is a Rudin-Keisler projection of $\sum_{U,\alpha}U_\alpha$ to $lim_UU_\alpha$.\end{fact}
 \begin{fact}\label{fact:galvin downward closed}
 If $U\leq_{RK} W$, then $W$ is Galvin $\Rightarrow$ $U$ is Galvin.
 \end{fact}
 \begin{fact}\label{Factfactor}
     If $U$ is an ultrafilter over $\kappa$ and $U_\alpha$ are ultrafilters for $\alpha<\kappa$ then $j_{\sum_{U,\alpha}U_\alpha}:V\rightarrow M_{\sum_{U,\alpha}U_\alpha}$ can be factored to $j^{M_U}_{U^*}\circ j_U$ where $U^*=[\alpha\mapsto U_\alpha]_U$.
 \end{fact}
 In the next theorem we will need the so-called \textit{modified diagonal intersection} which is defined as follows:
     Let $U$ be a $\kappa$-complete ultrafilter over $\kappa$ and let $\pi:\kappa\rightarrow\kappa$ be the minimal non constant function $(mod \ U)$, namely $[\pi]_U=\kappa$. For a sequence of sets $\l A_\alpha\mid \alpha<\kappa\r\subseteq U$, we defined the modified diagonal intersection by
     $$\Delta^{*,\pi}_{\alpha<\kappa}A_\alpha=\{\nu<\kappa\mid \forall \alpha<\pi(\nu), \ \nu\in A_\alpha\}$$
    We will ignore the superscript $\pi$ when the choice of $\pi$ is irrelevant or clear from the context. Note that if $U$ is normal, we may choose $\pi=id$ and then $\Delta^*_{\alpha<\kappa}A_\alpha=\Delta_{\alpha<\kappa}A_\alpha$ is the standard diagonal intersection. In general, it takes a routine argument to check that:
    \begin{fact}
        Let $U$ be a $\kappa$-complete ultrafilter over $\kappa$. For every $\l A_\alpha\mid \alpha<\kappa\r\subseteq U$, $\Delta^*_{\alpha<\kappa}A_\alpha\in U$.\end{fact}
 The following theorem is a joint result with M. Gitik:
\begin{theorem}\label{limit of ppoints}
 Suppose that $W\equiv_{RK}\sum_{U,\alpha} U_\alpha$ is such that $U,U_\alpha$ are $p$-point ultrafilters over $\kappa$. Then $W$ has the Galvin property. 
\end{theorem}
\begin{proof}
Let $\l A_i\mid i<\kappa^+\r\in [W]^{\kappa^+}$, denote by $A_{i,\alpha}^{(1)}=\{\beta\mid \l\alpha,\beta\r\in A_i\}$. Also, let $A^{(0)}_{i}=\{\alpha\mid A_{i,\alpha}^{(1)}\in U_\alpha\}\in U$. For every $\alpha_1<\alpha_2\in [\kappa]^{2}$, and every $\xi<\kappa^+$, define
$$H_{\xi,\l\alpha_1,\alpha_2\r}=\Big\{\gamma<\kappa^+\mid  A^{(0)}_{\gamma}\cap \alpha_1= A^{(0)}_{\xi}\cap \alpha_1\text{ and } \forall \beta<\alpha_1, \ A^{(1)}_{\gamma,\beta}\cap\alpha_2=A^{(1)}_{\xi,\beta}\cap\alpha_2\Big\}.$$
\begin{claim}\label{Claim: Glavin's claim}
There is $\xi^*<\kappa^+$, such that for every $\alpha_1,\alpha_2$, $|H_{\xi^*,\l\alpha_1,\alpha_2\r}|=\kappa^+$.
\end{claim}
\begin{proof}[\textit{Proof of Claim \ref{Claim: Glavin's claim}.}]
Suppose otherwise, then pick for each $\xi<\kappa^+$, $\alpha_{1,\xi},\alpha_{2,\xi}<\kappa$ such that the cardinality is at most $\kappa$. Stabilize these values on a set $X$ of size $\kappa^+$ with the values $\alpha_1^*,\alpha_2^*$.
The set $H_{\xi,\l \alpha_1^*,\alpha_2^*\r}$ is determined from $ A^{(0)}_{\xi}\cap\alpha_1^*$ and $\l A^{(1)}_{\xi,\beta}\cap \alpha^*_2\mid \beta<\alpha^*_1\r$. Since there are less than $\kappa$ many such sequences, we have a set $X^*$ of size $\kappa^+$ such that for any $\xi\in X^*$, $H_{\xi,\l\alpha^*_1,\alpha^*_2\r}=H$ but then $X^*\subseteq H$, contradiction.
\end{proof}
 Let $\xi^*<\kappa^+$ be as in  Claim \ref{Claim: Glavin's claim}. Since all the ultrafilters are $p$-points, for each $T\in \{U\}\cup\{U_\alpha\mid\alpha<\kappa\}$ there is a set $B^{(T)}\in T$ such that for every $j<\kappa$, there is $\rho^{(j)}_{T}$, such that $$\sup(\pi_{T}^{-1}[j+1]\cap B^{(T)}),j<\rho^{(j)}_{T}<\kappa$$ where $[\pi_T]_T=\kappa$. For $\alpha,j<\kappa$, let $\delta^{(j)}_\alpha=\sup( \rho^{(j)}_{U_\beta}\mid \beta<\alpha)<\kappa$.
Define the sequence $\l\beta_j\r_{j<\kappa}$ by induction, $$\beta_j\in H_{\xi^*,\l \rho^{(j)}_U,\delta^{(j)}_{\rho^{(j)}_U}\r}\setminus\{\beta_k\mid k<j\}.$$
The following claim will complete the proof of the theorem.\begin{claim}\label{claim: the galvin witness} $\bigcap_{j<\kappa} A_{\beta_j}\in \sum_{U,\alpha} U_\alpha$.\end{claim}
\begin{proof}[\textit{Proof of Claim \ref{claim: the galvin witness}}.] Set $$C^{(0)}:=A^{(0)}_{\xi^*}\cap \Delta^{*,\pi_U}_{j<\kappa} A^{(0)}_{\beta_j}\cap B^{(U)}\in U.$$
and for each $\beta\in C^{(0)}$, and each $j<\kappa$ we let $$B^{(1)}_{\beta_j,\beta}=\begin{cases} A^{(1)}_{\beta_j,\beta} & \beta\in A^{(0)}_{\beta_j}\\
\kappa & o.w.\end{cases}$$
Hence for every $j<\kappa$, $B^{(1)}_{\beta_j,\beta}\in U_\beta$. Define
$$C_\beta=\Big(A^{(1)}_{\xi^*,\beta}\cap \Delta^{*}_{j<\kappa} B^{(1)}_{\beta_j,\beta}\cap B^{(U_\beta)}\Big)\setminus \rho^{(\beta)}_{U_{\beta}}+1\in U_\beta.$$
Note here that the modified diagonal intersection is taken with respect to the function $\pi_{U_\beta}$.
Finally we let $C^*=\bigcup_{\beta\in C^{(0)}}\{\beta\}\times C_\beta$. Clearly, $C^*\in \sum_{U,\beta} U_\beta$. So it suffices to see that $C^*\subseteq \bigcap_{j<\kappa}A_{\beta_j}$.

Let $\l \alpha_1,\alpha_2\r\in C^*$, then $\alpha_1\in C^{(0)}$ and $\alpha_2\in C_{\alpha_1}$. By definition of $C_{\alpha_1}$, $\alpha_2>\rho^{(\alpha_1)}_{U_{\alpha_1}}$, and $\alpha_2\in B^{(U_{\alpha_1})}$. It follows that $$(*) \ \ \ \pi_{U_{\alpha_1}}(\alpha_2)>\alpha_1\geq \pi_U(\alpha_1).$$ Let $j<\kappa$, and let us split into cases:
\begin{enumerate}
    \item If $j<\pi_U(\alpha_1)$, then by the definition of modified diagonal intersection, $\alpha_1\in A^{(0)}_{\beta_j}$. By the definition of $B^{(1)}_{\beta_j,\alpha_1}$,  $B^{(1)}_{\beta_j,\alpha_1}=A^{(1)}_{\beta_j,\alpha_1}$. By $(*)$, $j<\pi_U(\alpha_1)<\pi_{U_{\alpha_1}}(\alpha_2)$ and therefore $\alpha_2\in B^{(1)}_{\beta_j,\alpha_1}=A^{(1)}_{\beta_j,\alpha_1}$. It follows that $\l \alpha_1,\alpha_2\r\in A_{\beta_j}$.
    \item If $\pi_U(\alpha_1)\leq j$, then $\alpha_1<\rho^{(j)}_U$, and so $\alpha_1\in C^{(0)}\cap \rho^{(j)}_U\subseteq A^{(0)}_{\xi^*}\cap\rho^{(j)}_U$. Since $\beta_j\in H_{\xi^*,\l\rho^{(j)}_U,\delta^{(j)}_{\rho^{(j)}_U}\r}$ it follows that $\alpha_1\in A^{(0)}_{\beta_j}\cap\rho^{(j)}_U.$ We conclude, once again, that $B^{(1)}_{\beta_j,\alpha_1}=A^{(1)}_{\beta_j,\alpha_1}$. Let us split again to cases according to the relation of $j$ and $\pi_{U_{\alpha_1}}(\alpha_2)$:
    \begin{enumerate}
        \item [(2a)] If $j<\pi_{U_{\alpha_1}}(\alpha_2)$, then by definition of the modified diagonal intersection, $\alpha_2\in B^{(1)}_{\beta_j,\alpha_1}=A^{(1)}_{\beta_j,\alpha_1}$. And again we conclude that $\l \alpha_1,\alpha_2\r\in A_{\beta_j}$.
        \item [(2b)] If $\pi_{U_{\alpha_1}}(\alpha_2)\leq j$, then by definition of $\rho^{(j)}_{U_{\alpha_1}}$, $\alpha_2<\rho^{(j)}_{U_{\alpha_1}}$ and $$(**) \ \ \ \alpha_2\in C_{\alpha_1}\cap \rho^{(j)}_{U_{\alpha_1}}\subseteq A^{(1)}_{\xi^*,\alpha_1}\cap \rho^{(j)}_{U_{\alpha_1}}. 
$$ Also by $(*)$ and the choice of $\rho^{(j)}_U$, $\alpha_1<\pi_{U_{\alpha_1}}(\alpha_2)\leq j<\rho^{(j)}_{U}$ and hence $\rho^{(j)}_{U_{\alpha_1}}
\leq\delta^{(j)}_{\rho^{(j)}_U}$ . Since $\beta_j\in H_{\xi^*,\l \rho^{(j)}_U,\delta^{(j)}_{\rho^{(j)}_U}\r}$, we conclude that for every $\beta<\rho^{(j)}_U$ (and in particular for $\alpha_1$) $A^{(1)}_{\xi^*,\beta}\cap \delta^{(j)}_{\rho^{(j)}_U}=A^{(1)}_{\beta_j,\beta}\cap \delta^{(j)}_{\rho^{(j)}_U}$.
It follows from $(**)$ that $\alpha_2\in A^{(1)}_{\beta_j,\alpha_1}\cap \delta^{(j)}_{\rho^{(j)}_{U}}$, hence $\l\alpha_1,\alpha_2\r\in A_{\beta_j}$. 
    \end{enumerate}
\end{enumerate}  We conclude that in any case $\l\alpha_1,\alpha_2\r \in A_{\beta_j}$.
This establishes that $C^*\subseteq \bigcap_{j<\kappa}A_{\beta_j}$. \end{proof}$\qedhere_{\text{theorem }\ref{limit of ppoints}}$
\end{proof}
\begin{corollary}\label{cor ppoint on kappa}
 Suppose $W\equiv_{RK}lim_U U_\alpha$ where $U,U_\alpha$ are all $p-$points on $\kappa$, then $W$ is Galvin.
 \end{corollary}
 \begin{proof}
Follows directly from Facts \ref{fact:RK projection of sum},\ref{fact:galvin downward closed}. \end{proof}

The argument in Theorem \ref{limit of ppoints} generalizes to ultrafilters which are obtained by taking a $p$-point sum of $p$-point sums of $p$-points, or a $p$-point sum of $p$-point sums of $p$-point sums of $p$-points. This goes on finitely many stages. Let us be more precise and define what an \textit{$n$-fold sum of $p$-points} is:
\begin{definition}\label{Def: n-fold}
     Let $W$ be a $\kappa$-complete ultrafilter over $\kappa$. $W$ is called a $1$-fold sum of $p$-points over $\kappa$ if $W$ is a $p$-point ultrafilter over $\kappa$. We proceed recursively on $1\leq n<\omega$, we say that $W$ is an $n+1$-fold sum of $p$-points over $\kappa$ if $W$ is a $\kappa$-complete ultrafilter over $\kappa$ which is Rudin-Keisler equivalent to an ultrafilter of the form $\sum_{U,\alpha}U_\alpha$, where $U$ is an $n$-fold sum of $p$-points over $\kappa$, and for each $\alpha$, $U_\alpha$ is a $p$-point over $\delta_\alpha\leq\kappa$. We say that $W$ is an \textit{iterated sum of $p$-points over $\kappa$} if there is $1\leq n<\omega$ such that $W$ is an $n$-fold sum of $p$-points over $\kappa$.
\end{definition}
Note that $W$ being an iterated sum of $p$-points over $\kappa$ is equivalent to $W$ being Rudin-Keisler equivalent to an ultrafilter taking the form $$\sum_{U,\alpha_1}\sum_{U_{\alpha_1},\alpha_2}...\sum_{U_{\alpha_1,...,\alpha_{n-2}},\alpha_{n-1}}U_{\alpha_1,...,\alpha_{n-1}}$$ where $U$ is a $p$-point over $\kappa$ and each $U_{\alpha_1,...,\alpha_k}$ is a $p$-point over some $\delta_{\alpha_1,..,\alpha_k}\leq\kappa$ for every $1\leq k\leq n$. If moreover all the $\delta_{\alpha_1,...,\alpha_n}=\kappa$, then we say that $W$ is a \textit{simple iterated sum of $p$-points over $\kappa$}.

The proof of the following theorem is completely analogous to the one of Theorem $\ref{limit of ppoints}$ with only the notations being (even) more complicated.
\begin{theorem}\label{Theorem: general n-fold of kappa p-points}
    Suppose that $W$ is a simple iterated sum of $p$-points over $\kappa$. Then $W$ has the Galvin property. 
\end{theorem}
\begin{proof}
     Since the Galvin property is invariant under Rudin-Keisler equivalence, and $W$ is a simple iterated sum of $p$-points,  we may assume that $W$ take the form $$W=\sum_{U,\alpha_1}\sum_{U_{\alpha_1},\alpha_2}...\sum_{U_{\alpha_1,...,\alpha_{n-2}},\alpha_{n-1}}U_{\alpha_1,...,\alpha_{n-1}}$$
     where each $T\in\{U\}\cup\{U_{\alpha_1,...,\alpha_k}\mid 0\leq k\leq n-1, \ \alpha_1,...,\alpha_k<\kappa\}$ is a $p$-point ultrafilter over $\kappa$. 
     
    By induction, one proves that $W$ is an ultrafilter over $[\kappa]^n$ and for every $X\subseteq [\kappa]^n$
    $X\in W$ iff
    $$(*) \ \ \ \ \Big\{\alpha_1<\kappa\mid \{\alpha_2<\kappa\mid...\{ \alpha_n<\kappa\mid \l \alpha_1,...,\alpha_n\r\in X\}\in U_{\alpha_1,...,\alpha_{n-1}}...\}\in U_{\alpha_1,\alpha_2}\}\in U_{\alpha_1}\Big\}\in U.$$
 Let $\l A_i\mid i<\kappa^+\r\in [W]^{\kappa^+}$. For any $i<\kappa^+$, and $\l\alpha_1,...,\alpha_{n-1}\r\in [\kappa]^{n-1}$, we denote $$A^{(n-1)}_{i,\alpha_1,...,\alpha_{n-1}}=\{\alpha\mid \l\alpha_1,...,\alpha_{n-1},\alpha\r\in A_i\},$$
 $$A^{(n-2)}_{i,\alpha_1,...,\alpha_{n-2}}=\{\alpha\mid A^{(n-1)}_{i,\alpha_1,...,\alpha_{n-2},\alpha}\in U_{\alpha_1,...,\alpha_{n-2},\alpha}\}.$$
 We proceed recursively to define for each $0\leq k\leq n-2$,
 $$A^{(k)}_{i,\alpha_1,...,\alpha_k}=\{\alpha\mid A^{(k+1)}_{i,\alpha_1,...,\alpha_k,\alpha}\in U_{\alpha_1,...,\alpha_k,\alpha}\}.$$
 For notational reasons, the sequence $\l \alpha_1,...,\alpha_{k}\r$ where $k=0$ is just the empty sequence. In particular, we have defined above 
 $$A^{(0)}_{i}=\{\alpha\mid A^{(1)}_{i,\alpha}\in U_\alpha\}.$$
 Note that the set in $(*)$ is simply $A^{(0)}_{i}$. For every $\xi<\kappa^+$ and $\alpha<\kappa$ define
 $$H_{\xi,\alpha}=\{\gamma<\kappa^+\mid \forall 0\leq k\leq n-1, \forall \alpha_1,...,\alpha_k<\alpha, \ A^{(k)}_{\gamma,\alpha_1,..,\alpha_k}\cap\alpha=A^{(k)}_{\xi,\alpha_1,..,\alpha_k}\cap\alpha\}$$
 Then as in Claim \ref{Claim: Glavin's claim}, we can prove that there is $\xi^*<\kappa^+$ such that for every $\alpha$, $|H_{\xi^*,\alpha}|=\kappa^+$. 
 For each $T\in \{U\}\cup \{U_{\alpha_1,...,\alpha_k}\mid 0\leq k\leq n-1, \ \alpha_1,...,\alpha_k<\kappa\}$ we define $\rho^{(j)}_T$ as in the proof of Theorem \ref{limit of ppoints}.
 
 Then we define $$\delta_j=\sup\{\rho^{(j)}_{U_{\alpha_1,...,\alpha_k}}\mid 0\leq k\leq n-1, \  \alpha_1,...,\alpha_k\leq j\}$$
 
 Now choose as before $\beta_j\in H_{\xi^*,\delta_{j}}\setminus\{\beta_i\mid i<j\}$.
 Finally we show that $\bigcap_{j<\kappa}A_{\beta_j}\in W$ by constructing a subset of it $C^*$, which is in $W$. The definition of $C^{(0)}$ is exactly as before. 
 In general, suppose that $\gamma_1\in C^{(0)},\gamma_2\in C^{(1)}_{\gamma_1}$,... $\gamma_{k+1}\in C^{(k)}_{\gamma_1,...,\gamma_{k}}$, and $0\leq k<n-1$ define first $$B^{(k+1)}_{\beta_j,\gamma_1,...,\gamma_{k+1}}=\begin{cases}
     A^{(k+1)}_{\beta_j,\gamma_1,...,\gamma_{k+1}} & \gamma_{k+1}\in A^{(k)}_{\beta_j,\gamma_1,...,\gamma_{k}}\\
     \kappa & o.w.\end{cases}$$
and $r_{\gamma_1,...,\gamma_{k+1}}=\max\{\rho^{\gamma_i}_{U_{\gamma_1,...,\gamma_{k+1}}}\mid 1\leq i\leq k+1\}$. Then $$C^{(k+1)}_{\gamma_1,...,\gamma_{k+1}}=\big(A^{(k+1)}_{\xi^*,\gamma_1,...,\gamma_{k+1}}\cap \Delta^*_{j<\kappa}B^{(k+1)}_{\beta_j,\gamma_1,...,\gamma_{k+1}}\cap B^{(U_{\gamma_1,..,\gamma_{k+1}})}\big)\setminus r_{U_{\gamma_1,...,\gamma_{k+1}}}+1\in U_{\gamma_1,...,\gamma_{k+1}}$$
The set $C^*$ is define as
$$C^*=\bigcup_{\gamma_1\in C^{(0)}}\bigcup_{\gamma_2\in C^{(1)}_{\gamma_1}}...\bigcup_{\gamma_{n-1}\in C^{(n-2)}_{\gamma_1,..,\gamma_{n-2}}}\{\gamma_1\}\times\{\gamma_2\}\times...\times \{\gamma_{n-1}\}\times C^{(n-1)}_{\gamma_1,...,\gamma_{n-1}}$$
Then we conclude that $C^*\in W$. Now take any $\l\gamma_1,...,\gamma_{n}\r\in C^*$, then $\gamma_n\in C^{(n-1)}_{\gamma_1,...,\gamma_{n-1}}$ and for every $1\leq k\leq n$, $\gamma_{k}\in C^{(k-1)}_{\gamma_1,...,\gamma_{k-1}}$. 
It follows that
$\gamma_k>r_{\gamma_1,...,\gamma_{k-1}}\geq\rho^{\gamma_{i}}_{U_{\gamma_1,...,\gamma_{k-1}}}$ for all $1\leq i\leq k-1$. Hence $$(**)\ \ \pi_{U_{\gamma_1,...,\gamma_{k-1}}}(\gamma_k)>\gamma_1,...,\gamma_{k-1}.$$

We prove inductively on $1\leq i\leq n$, that for every $j<\kappa$, $\gamma_i\in A^{(i-1)}_{\beta_j,\gamma_1,...,\gamma_{i-1}}$. This will implies that $\l \gamma_1,...\gamma_n\r\in\bigcap_{j<\kappa}A_{\beta_j}$. 

For $i=1$, let $j<\kappa$. If $j<\pi_U(\gamma_1)$, then by the definition of modified diagonal intersection, $\gamma_1\in A^{(0)}_{\beta_j}$. If $j\geq \pi_U(\gamma_1)$, then $\gamma_1<\rho^{(j)}_U\leq\delta_j$. In which case $\gamma_1\in A^{(0)}_{\xi^*}\cap \delta_j$, and since $\beta_j\in H_{\xi^*,\delta_j}$, $A^{(0)}_{\xi^*}\cap \delta_j=A^{(0)}_{\beta_j}\cap \delta_j$ which in turn implies that $\gamma_1\in A^{(0)}_{\beta_j}$.

Now suppose that $\gamma_{k}\in A^{(k-1)}_{\beta_j,\gamma_1,...,\gamma_{k-1}}$. In particular, by definition, $B^{(k)}_{\beta_j,\gamma_1,...,\gamma_k}=A^{(k)}_{\beta_j,\gamma_1,...,\gamma_k}$. Let us prove that $\gamma_{k+1}\in A^{(k)}_{\beta_j,\gamma_1,...,\gamma_k}$. Again, let us split into cases. If $j<\pi_{U_{\gamma_1,...\gamma_k}}(\gamma_{k+1})$, by the definition of the modified diagonal intersection, $\gamma_{k+1}\in B^{(k)}_{\beta_j,\gamma_1,...,\gamma_k}=A^{(k)}_{\beta_j, \gamma_1,...,\gamma_k}$. So suppose that $j\geq \pi_{U_{\gamma_1,...,\gamma_k}}(\gamma_{k+1})$, this means that $\gamma_{k+1}<\rho^{(j)}_{U_{\gamma_1,...,\gamma_k}}$. By $(**)$, $\gamma_1,...,\gamma_k<\pi_{U_{\gamma_1,...,\gamma_k}}(\gamma_{k+1})\leq j$. By definition of $\delta_j$, $\rho^{(j)}_{U_{\gamma_1,...,\gamma_k}}\leq\delta_j$, hence $\gamma_{k+1}\in A^{(k)}_{\xi^*,\gamma_1,...,\gamma_k}\cap \delta_j$. Since $\beta_j\in H_{\xi^*,\delta_j}$, and since $\gamma_1,...,\gamma_k<\delta_j$, $A^{(k)}_{\beta_j,\gamma_1,...,\gamma_k}\cap \delta_j=A^{(k)}_{\xi^*,\gamma_1,...,\gamma_k}\cap \delta_j$, and therefore $\gamma_{k+1}\in A^{(k)}_{\beta_j,\gamma_1,...,\gamma_k}$.
\end{proof}
 
 Next, we deal with non-simple iterated sums of $p$-points. 
 \begin{lemma}\label{Lemma: when a sum is a p-point}
 Suppose that $U$ is a $p$-point over $\kappa$, $\pi_U:\kappa\rightarrow\kappa$ is a function such that $[\pi_U]_U=\kappa$, and 
 for every $\alpha<\kappa$, $U_\alpha$ is a $p$-point ultrafilter over $\delta_\alpha$, where $\pi_U(\alpha)<\delta_\alpha<\kappa$. Then $\sum_{U,\alpha}U_\alpha\equiv_{RK} W$ where $W$ is a $p$-point ultrafilter.
 \end{lemma}
 \begin{proof}
 Let $W$ be any ultrafilter over $\kappa$ such that $\sum_{U,\alpha}U_\alpha\equiv_{RK}W$, let us prove that $W$ is a $p$-point. Note that by Rudin-Keisler equivalence, $M_W=M_{\sum_{U,\alpha}U_\alpha}$ and $j_W=j_{\sum_{U,\alpha}U_\alpha}$. Let $U^*=[\alpha\mapsto U_\alpha]_U$, then $U^*\in M_U$ and $U^*$ is an ultrafilter over $$\kappa=[\pi_U]_U<\kappa^*=[\alpha\mapsto \delta_\alpha]_U<j_U(\kappa).$$ By fact \ref{Factfactor}, $j_W=j_{\sum_{U,\alpha}U_\alpha}$ can be factored into $j_{U^*}\circ j_U$. Now since $j_U(\kappa)$ is a measurable cardinal in $M_U$, and $\kappa^*=crit(j_{U^*})<j_U(\kappa)$, it follows that $j_W(\kappa)=j_{U^*}(j_U(\kappa))=j_U(\kappa)$. To see that $W$ is a $p$-point, let us use Lemma \ref{kanamoriequiv}, let $f:\kappa\rightarrow\kappa$ be any function such that $[f]_W<j_W(\kappa)=j_U(\kappa)$. Since $U$ is a $p$-point, find some monotone function $g:\kappa\rightarrow \kappa$ such that $[f]_W\leq j_U(g)(\kappa)$. It follows that $[f]_W\leq j_{U^*}(j_U(g)(\kappa))=j_W(g)(j_{U^*}(\kappa))$. Finally, since $\kappa^*>\kappa$ it follows that $j_{U^*}(\kappa)=\kappa$, so $[f]_W\leq j_W(g)(\kappa)$ as wanted.  
 \end{proof}
 \begin{theorem}\label{general ppoint}
Let $W$ be a $\kappa$-complete ultrafilter over $\kappa$. Suppose that $W\equiv_{RK}\sum_{U,\alpha} U_\alpha$ where $U$ is a $p$-point ultrafilter over $\kappa$ and for each $\alpha<\kappa$, $U_\alpha$ is a $p$-point ultrafilter over some $\delta_\alpha\leq\kappa$. Then $W$ has the Galvin property. 
\end{theorem}
\begin{proof}
Let $\pi_U:\kappa\rightarrow\kappa$ be a function such that $[\pi_U]_U=\kappa$. First note that there is a set $X\in U$ such that exactly one of the following holds:
 \begin{enumerate}
     \item For each $\alpha\in X$, $\delta_\alpha=\kappa$.
     \item For each $\alpha\in X$, $\pi_U(\alpha)<\delta_\alpha<\kappa$.
     \item For each $\alpha\in X$, $\pi_U(\alpha)=\delta_\alpha$.
 \end{enumerate}
 So let us split into cases according to the above:
\begin{enumerate}
    \item If for every $\alpha<\kappa$, $\delta_\alpha=\kappa$,   apply Theorem \ref{cor ppoint on kappa}.
    \item If for every $\alpha<\kappa$, $\pi_U(\alpha)<\delta_\alpha<\kappa$,  apply Lemma \ref{Lemma: when a sum is a p-point}.
    \item If for every $\alpha<\kappa$ $\pi_U(\alpha)=\delta_\alpha$, then $U^*=[\alpha\mapsto U_\alpha]_{U}$ is a $p$-point ultrafilter over $[\pi_U]_U=\kappa$ in $M_U$. Since $M_U$ is closed under $\kappa$-sequences, it follows that $U^*$ is $p$-point over $\kappa$ in $V$. So $\sum_{U,\alpha} U_\alpha\equiv_{RK} U^*\times U$, and by Theorem \ref{galvinGen}, a product of $p$-points is Galvin.
\end{enumerate}    
\end{proof}
The argument above generalizes to iterated sums of $p$-points over $\kappa$:
\begin{theorem}\label{generalgeneral p-point}
    Let $W$ be a $\kappa$-complete ultrafilter over $\kappa$ which is an $n$-fold sum of $p$-points of ultrafilters over cardinals which are possibly less than $\kappa$. Then $W$ has the Galvin property. 
\end{theorem}
\begin{proof} The proof is by induction on $n$. Let $W$ be a $\kappa$-complete $n+1$-fold of $p$-points (not necessarily all over $\kappa$). First, note that $j_W$ can be factored as a finite iterated ultrapower $$V = M_0\overset{j_{0,1}}{\longrightarrow}M_1\overset{j_{1,2}}{\longrightarrow}\cdots\overset{j_{n,n+1}}{\longrightarrow}M_{n+1}=M_W$$
        where $j_{0,1}:V\rightarrow M_1$ is the ultrapower embedding of a $p$-point ultrafilter $U_0$ over $\kappa=\kappa_0$ and for every $1\leq k\leq n$, $j_{k,k+1}:M_k\rightarrow M_{k+1}$ is the ultrapower embedding associated to some $p$-point ultrafilter \(U_k\in M_{k}\) over some $\delta^*_k$, where  $\kappa\leq \delta^*_k\leq \kappa_{k}=j_{0,k}(\kappa)$. By rearranging the ultrafilters $U_k$, and by case $(3)$ in the proof of Theorem \ref{general ppoint}, we may assume that the sequence of critical points $\delta^*_k$ is strictly increasing. Let us split the proof into cases: 
        \begin{enumerate}
            \item If there is $0<k\leq n$ such that $\delta^*_{k-1}<\delta^*_k<\kappa_k$,  let $k_0$ be the least such $k$. Then, $\kappa_{k_0-1}=\delta^*_{k_0-1}<\delta^*_{k_0}<\kappa_{k_0}=j_{U_{k_0-1}}(\kappa_{k_0-1})$. Apply Lemma \ref{Lemma: when a sum is a p-point} inside $M_{k_0-1}$ and deduce that the two-step iteration $j_{k_0,k_0+1}\circ j_{k_0-1,k_0}=j_U$ for some $p$-point ultrafilter $U\in M_{k_0-1}$ over $\kappa_{k_0-1}$. The conclusion is that we can reduce $j_W$ to an $n$-step iteration and apply the induction hypothesis.
        \item  Otherwise, for each $k$, $\delta^*_k=\kappa_k$, in which case $W$ is an iterated sum of $p$-points over $\kappa$. Hence we may apply Theorem \ref{Theorem: general n-fold of kappa p-points} to conclude that $W$ is a Galvin ultrafilter. \end{enumerate} 
\end{proof}
\section{Galvin ultrafilters in canonical inner models}
In this section, we wish to apply the results from the previous section to study the Galvin property in canonical inner models. The first result in this direction was done in \cite{TomMotiII}:
\begin{proposition}
 Let $U$ be a normal measure over $\kappa$. Then in $L[U]$, every $\kappa$-complete (even $\sigma$-complete) ultrafilter is a Galvin ultrafilter.
\end{proposition}
The reason for this is that every $\kappa$-complete ultrafilter in $L[U]$ is Rudin-Keisler equivalent to a finite power of the normal measure $U$, and then theorem \ref{galvinGen} applies to such ultrafilters. Our intention is to use a similar argument to prove the same result for canonical inner models suitable for larger cardinals. It turns out that the results from the previous section can be used to prove that in every canonical inner model up to a superstrong cardinal, every $\kappa$-complete ultrafilter is Galvin. 

In order to prove that, we will use Goldberg's recent work on the \textit{Ultrapower Axiom} ($\mathrm{UA}$) \cite{GoldbergUA} and use his results regarding the structure of ultrafilters under this axiom. 

Let us start this section  by stating Goldberg's results which are going to be used here. First, and most importantly, the fact the $\mathrm{UA}$ follows from Weak Comparison, and therefore should hold in all canonical inner models: 
\begin{theorem}[Theorem 2.3.10. \cite{GoldbergUA}]
Assume that $V=HOD$ and there is a $\Sigma_2$-correct worldly cardinal. If
Weak Comparison holds, then the Ultrapower Axiom holds
\end{theorem}
The $\mathrm{UA}$ has many consequences which relate to the structure of $\kappa$-complete ($\sigma$-complete) ultrafilters. The one which is most relevant to our work is the factorization of $\sigma$-complete ultrafilters to irreducible ultrafilters.
\begin{definition}
Let $U,W$ be $\sigma$-complete ultrafilters. We define the Rudin-Frol\'{i}k ordering of ultrafilter by $U\leq_{RF}W$ if there is there is a set $I \in U$ and a discrete
sequence\footnote{A sequence of ultrafilters $\l W_i\r_{i<\kappa}$ is discrete if there pairwise disjoint sets $\l A_i\r_{i<\kappa}$ such that $A_i\in W_i$.} of ultrafilters $\l W_i
: i \in I\r$ such that $W \equiv_{RK} lim_U W_i$
.
\end{definition}
Note that since the sequence of $W_i$ is discrete, we may replace in the definition above $\lim_UW_i$ by $\sum_{U,i}W_i$.
There is an equivalent formulation (for $\sigma$-complete ultrafilters) in terms of ultrapowers, $U\leq_{RF}W$ if and only if there is an internal ultrapower embedding $i:M_U\rightarrow M_W$ such that $i\circ j_U=j_W$.
\begin{definition}
 A $\sigma$-complete ultrafilter $W$ is called \textit{irreducible}, if whenever $U\leq_{RF}W$ then either $U\equiv_{RK} W$ or $U$ is principle. Equivalently, $W$ is minimal in the Rudin-Frol\'{i}k order among the non-principal ultrafilters.
\end{definition}
The notion of irreducible ultrafilters was used by Goldberg to describe the $\sigma$-complete ultrafilters under $\mathrm{UA}$:
\begin{theorem}[Theorem 5.3.16\cite{GoldbergUA}]\label{thm: Goldberg} Assume $\mathrm{UA}$.
 Then for every $\sigma$-complete ultrafilter $W$, there is a
finite linear iterated ultrapower 
$\l M_n, U_m, j_{m,n} : m <n \leq l\r$ such that $M_0 = V , M_l = M_W$ , and
$U_m$ is an irreducible ultrafilter of $M_m$ for all $m < l$, and $j_W = j_{0,l}$.
\end{theorem}
\begin{corollary}\label{Cor: Gabe decomposition of RF}
 Assume $\mathrm{UA}$, then every $\sigma$-complete ultrafilter $W$ has the form $$W=\sum_{U,\alpha_1}(\sum_{U_{\alpha_1},\alpha_2}\cdots(\sum_{U_{\alpha_1,...,\alpha_{n-1}},\alpha_n}U_{\alpha_1,...,\alpha_{n}})))$$
 where each $U_{\alpha_1,..,\alpha_k}$ is irreducible.
\end{corollary}
The $\mathrm{UA}$ assumption in Theorem \ref{thm: Goldberg} and Corollary \ref{Cor: Gabe decomposition of RF} can be replaced by working in the Mitchell-Steel models\footnote{In this paper, we use the Mitchell-Steel indexing, and allow extenders of superstrong type in the sequence $\mathbb{E}$. These models fall under the setup of \cite{Schlutz1}.} of the form $L[\mathbb{E}]$. This was essentially\footnote{Although in \cite{Schlutz} there is the extra assumption that there are no extenders of superstrong-type, the only significant difference between the argument in \cite[Thm. 4.8]{Schlutz} and the result in \cite[Thm. 9.1]{Schlutz1} without the extra assumption, is that all extenders on the sequence are Dodd-solid.  Moreover, variants of this were already used by Zeman \cite{ZemanIndex}.} proven by Schlutzenberg in \cite{Schlutz}, replacing ``irreducible" with ``on the sequence $\mathbb{E}$", before Goldberg's discovery of $\mathrm{UA}$. The next theorem, also due to Schlutzenberg \cite[Proposition 8.5]{Schlutz1} says that in the Mitchell-Steel models $L[\mathbb{E}]$, these notions of ultrafilters coincide (see also \cite[Theorem 4.3.2]{GoldbergUA}).
\begin{theorem}\label{schluz}
 Suppose $L[\mathbb{E}]$ is an iterable Mitchell-Steel model and $U$ is
a $\sigma$-complete ultrafilter of $L[\mathbb{E}]$. Then the following are equivalent:
\begin{enumerate}
    \item $U$ is irreducible.
    \item $U$ is isomorphic to a Dodd-sound\footnote{An ultrafilter $U$ over $\kappa$ is called Dodd-sound if the function $i:P(\kappa)\rightarrow P^{M_U}([id]_U)$ defined by $i(X)=j_U(X)\cap [id]_U$ is in $M_U$.} ultrafilter.
\item $U$ is isomorphic to an extender on the sequence $\mathbb{E}$.
\end{enumerate}

\end{theorem}
There is a slight ambiguity in Goldberg's book \cite{GoldbergUA} regarding the citation of this theorem. With the gratitude of Goldberg, and thanks to the referee, let us clarify the correct citation for this result: First, $(2)\Rightarrow(1)$ is a simple folklore $\mathrm{ZFC}$ implication (see for example \cite[Prop. 5.3.6]{GoldbergUA}). If $L[\mathbb{E}]$ is the Mitchell-Steel model without extenders of superstrong type (as in the original Mitchell-Steel's \cite{Mitchell2017FineSA}) then $(1)\Rightarrow (3)$ is a direct consequence of \cite[Thm. 4.8]{Schlutz}, then without the superstrong restriction, this implication is a consequence of \cite[Thm. 9.1]{Schlutz1}, both due to Schlutzenberg\footnote{It seems to be open whether $(1)\Rightarrow (3)$ hold for the $\lambda$-indexing version of $L[\mathbb{E}]$.}. As for $(3)\Rightarrow(2)$, the implication for Mitchell-Steel indexing without superstrong type extenders, is due to Steel (see \cite[\S~4]{SchindlerSteel}). The implication for Mitchell-Steel indexing without the superstrong restriction is due to Schlutzenberg \cite{Schlutz1}. For $\lambda$-indexed  $L[\mathbb{E}]$ models, which also have no restriction on superstrong extenders, this was proved earlier by Zeman \cite{ZemanInd}. As mentioned, we are interested in  $L[\mathbb{E}]$ constructed via the Mitchell-Steel indexing without the superstrong type restriction, hence the variation of Theorem \ref{schluz} we need is due to Schlutzenberg \cite{Schlutz1}.

 The next proposition is well-known. It is the last  ingredient in our main result of this section: 
\begin{proposition}\label{Prop: in L[E] irreducible is p-point}
Suppose that $U$ is a Dodd-sound ultrafilter over $\kappa$ which is not a $p$-point. Then $\kappa$ is a limit of superstrong cardinals.   
\end{proposition}
\begin{proof} Let $f$ be the $U$-minimal almost one-to-one function, and $\nu=[f]_U$. The minimality of $f$ guarantees that $\nu$ is a strong generator, namely, for every $g:\kappa\rightarrow\kappa$ and every $\rho<\nu$, $j_{U}(g)(\rho)<\nu$.
    Let $E$ be the $(\kappa,\nu)$-extender obtained derived from $j_U$, and $j_E:V\rightarrow M_U$ be the ultrapower embedding by $E$. Then\footnote{See for example the argument in Claim 2 of \cite[Thm. 8.27]{Steel2010inner}.} $j_E(\kappa)=\nu=crit(k)$ where $k:M_E\rightarrow M_{U}$ is the factor map. 
Since $U$ is Dodd-sound, and $\nu\leq[id]_U$ (by minimality of $f$), $E\in M_U$ and we can form $(j_E)^{M_U}:M_U\rightarrow (M_E)^{M_U}\equiv_{RK} Ult(M_U,E)$. Note that since $M_U$ is closed under $\kappa$-sequences, $(j_E)^{M_U}=j_E\restriction M_U$ and $$(M_E)^{M_U}=\{j_E(f)(a)\mid a\in[\nu]^{<\omega}, \  f:[\kappa]^{|a|}\rightarrow M_U\}.$$
     We claim that $(V_{\nu})^{M_U}\subseteq (M_E)^{M_U}$, which then implies that $M_U\models``\kappa$ is a superstrong cardinal", which then reflects to unboundedly many cardinals below $\kappa$. 
     To see that $(V_{\nu})^{M_U}\subseteq (M_E)^{M_U}$, first note that since $crit(k)=\nu$ and since $\nu=j_E(\kappa)$ is strongly inaccessible in $M_E$, $(V_{\nu})^{M_U}=(V_\nu)^{M_E}$. Hence if $x\in (V_\nu)^{M_U}$ then there is $a\in[\nu]^{<\omega}$ and $f:\kappa^{|a|}\rightarrow V_\kappa\subseteq M_U$ such that $x=j_E(f)(a)\in (M_E)^{M_U}$, as wanted.
     
\end{proof}
In particular, if there is no measurable limit of superstrong cardinals in the Mitchell-Steel model $L[\mathbb{E}]$, then 
 in $L[\mathbb{E}]$, every irreducible is a $p$-point. Combining this with Schlutzenberg's result (i.e. Corollary \ref{Cor: Gabe decomposition of RF} in $L[\mathbb{E}]$ - see the discussion following Corollary \ref{Cor: Gabe decomposition of RF}) we get: 
\begin{corollary}\label{FormofUlt}
    Suppose that there is no measurable limit of superstrong cardinals in the Mitchell-Steel model $L[\mathbb{E}]$. Then every $\kappa$-complete ultrafilter over $\kappa$ in $L[\mathbb{E}]$ is Rudin-Keisler isomorphic to an ultrafilter of the form $$\sum_{U,\alpha_1}\sum_{U_{\alpha_1},\alpha_2}...\sum_{U_{\alpha_1,...,\alpha_{n-1}},\alpha_n}(U_{\alpha_1,...,\alpha_n})$$ where each $U_{\alpha_1,...,\alpha_k}$ for $0\leq k\leq n$ is a $p$-point. 
\end{corollary}
Note that our result also applies to any model which satisfies the following three assumptions: $\mathrm{UA}$, ``every irreducible is Dodd-sound", and ``there is no measurable limit of superstrong cardinals" 
Finally, we obtain our main result of this section:
\begin{corollary}\label{GalinCan}
Suppose that there is no measurable limit of superstrong cardinals in the Mitchell-Steel model $L[\mathbb{E}]$. Then every $\kappa$-complete ultrafilter over $\kappa$ satisfies $Gal(U,\kappa,\kappa^+)$.
\end{corollary}
\begin{question}
Does $\mathrm{UA}$ imply that every $\kappa$-complete ultrafilter over $\kappa$ satisfies that Galvin property?
\end{question}
We conjecture that the answer is no. However, it is still interesting to ask how far in the hierarchy of inner models we have that every $\kappa$-complete ultrafilter over $\kappa$ satisfies the Galvin property.
\begin{question}
Is there a Galvin ultrafilter which is not Rudin-Keisler equivalent to an ultrafilter of the form appearing in corollary \ref{FormofUlt}?
\end{question}
In terms of ultrapowers, an ultrafilter $W$ which is not Rudin-Keisler equivalent to an ultrafilter of the form appearing in corollary \ref{FormofUlt}, should have infinitely many generators with respect to $j_W$. Alternatively, one can use Kanamori's terminology of skies and constellations from \cite{Kanamori1976UltrafiltersOA} and require that $W$ has infinitely many skies.
\section{Supercompact cardinals}
So far we have established that in the canonical inner models up to a superstrong, every $\kappa$-complete ultrafilter over $\kappa$ has the Galvin property. This shows that even for larger cardinals than just a measurable, there is no $ZFC$ construction of a $\kappa$-complete ultrafilter over $\kappa$ which is not Galvin. However, if we turn to the realm of large cardinals where we currently have no canonical inner models, this situation changes. Below we prove that there is a $ZFC$ construction for a non-Galvin ultrafilter over $\kappa$ assuming that $\kappa$ is a supercompact cardinal. This answers a question from \cite[Question 4.5]{BenGarShe} and the parallel one in \cite[Question 2.25]{bgp}. To prove it, we will use the notion of a $\kappa$-independent family:  Recall that a sequence of subsets of $\kappa$, $\l A_i\mid i<\lambda\r$ is called \textit{$\kappa$-independent} if for every $I,J\in[\lambda]^{<\kappa}$, if $I\cap J=\emptyset$ then $(\bigcap_{i\in I}A_i)\cap (\bigcap_{j\in I}A_j^c)\neq\emptyset$. One can show that if $\kappa^{<\kappa}=\kappa$, there is a $\kappa$-independent family of size $2^\kappa$
(see \cite[Exercise 8.10]{KunenBook}).
\begin{theorem}\label{SupercompactGalvin}
 Suppose that $j:V\rightarrow M$ is an elementary embedding with $crit(j)=\kappa$, $M^{\kappa}\subseteq M$, $j_U:V\rightarrow M_U$ is the derived normal ultrapower and $k:M_U\rightarrow M$ is the factor map. Suppose that $j[\kappa^+],k[j_U(\kappa^+)]\in M$ and that $j_U(\kappa^+)<j(\kappa)$. Then there is a non-Galvin $\kappa$-complete ultrafilter over $\kappa$.
\end{theorem}
\begin{proof}
Let $\l A_i\mid i<\kappa^+\r$ be a $\kappa$-independent family. 
Denote by $\l A'_\alpha\mid \alpha<j(\kappa^+)\r=j(\l A_i\mid i<\kappa^+\r)$. Since $j[\kappa^+],k[j_U(\kappa^+)]\in M$, in $M$ we have the sequence
$\l A'_{r}\mid r\in j[\kappa^+]\r$ and $\l A'_s\mid s\in k[j_U(\kappa^+)]\setminus j[\kappa^+]\r$. Note that the transitive collapse of $j_U[\kappa^+]$ is $\kappa^+$ and it belongs to $M$ and therefore $M\models |j[\kappa^+]|=\kappa^+<j(\kappa)$. Also, the transitive collapse of $k[j_U(\kappa^+)]$ is $j_U(\kappa^+)$, and by the assumptions of the 
theorem, $j_U(\kappa^+)<j(\kappa)$. It follows that  $M\models |k[j_U(\kappa^+)]|<j(\kappa)$. Next, the $j(\kappa)$-independence of the family $\l A'_\alpha\mid \alpha<j(\kappa^+)\r$ implies that there is $$\kappa\leq\delta\in (\bigcap_{r\in j[\kappa^+]}A'_r)\cap( \bigcap_{s\in k[j_U(\kappa^+)]\setminus j[\kappa^+]}(A'_s)^c).$$ Let $W=\{X\subseteq\kappa\mid \delta\in j(X)\}$. Then $\{A_i\mid i<\kappa^+\}\subseteq W$ and for every $I\in[\kappa^+]^\kappa$, consider $j_U(I)$. There is $r\in j_U(I)\setminus j_U[\kappa^+]$ e.g the $\kappa$-th elements in $j_U(I)$ in the increasing enumeration. Then $k(r)\in j(I)$ and not in $j[\kappa^+
]$, thus $\delta\notin A'_{k(r)}$. It follows that $\delta\notin \bigcap_{s\in j(I)}A'_s=j(\bigcap_{i\in I}A_i)$. This implies that $\bigcap_{i\in I}A_i\notin W$.
\end{proof}
\begin{remark}
In the previous theorem, we can replace the assumption $j_U(\kappa^+)<j(\kappa)$ with the assumption $j_U(\kappa)<j(\kappa)$. Indeed, for each $\alpha<j_U(\kappa^+)$, there is a well order $R_\alpha\in M_U$ of $j_U(\kappa)$ such that $\text{otp}(j_U(\kappa),R_\alpha)=\alpha$. Since $ \l k[j_U(\kappa)],k(R_\alpha)\restriction k[j_U(\kappa)]\r$ is order isomorphic to  $\l j_U(\kappa),R_\alpha\r$, we deduce that in $M$ there is a well ordering of $k[j_U(\kappa)]$ (and in turn of $j_U(\kappa)$) of order-type $\alpha$. Hence $\alpha<(j_U(\kappa)^+)^M$ and $j_U(\kappa^+)\leq (j_U(\kappa)^+)^M$. Since $j_U(\kappa)<j(\kappa)$ and $j(\kappa)$ is inaccessible in $M$, $j_U(\kappa^+)<j(\kappa)$.

Let us point out that we do not know it the assumption $j_U(\kappa)<j(\kappa)$ follows from the other assumptions. 
\end{remark}
Note that every $2^\kappa$-supercompact embedding satisfies the assumptions of Theorem 3.1. Hence, we conclude the following:
\begin{corollary}
 If $\kappa$ is $2^\kappa$-supercompact then there is  a non-Galvin ultrafilter over $\kappa$.
\end{corollary}
\begin{proposition}
 The first cardinal such that there is a non-Galvin ultrafilter is below the first supercompact.
\end{proposition}
\begin{proof}
 Let $\kappa$ be $2^\kappa$-supercompact, then by the previous corollary, $\kappa$ carries a $\kappa$-complete ultrafilter $U$ which is not Galvin. Let $j:V\rightarrow M$ be an embedding witnessing $2^\kappa$-supercompactness, then $U\in M$ and there is a witness $\l A_i\mid i<\kappa^+\r$ for $\neg Gal(U,\kappa,\kappa^+)$ which also belongs to $M$. Therefore, $$M\models ``\kappa\text{  carries a non-Galvin }\kappa\text{-complete ultrafilter}"$$
 and by reflection, this should hold at many cardinals below $\kappa$.
\end{proof}
\section{Non-Galvin ultrafilters Extending the club filter}
Note that the ultrafilter we constructed in the previous section does not necessarily extend the club filter. However, this was not the situation in the latest constructions of non-Galvin ultrafilters \cite{BenGarShe,OnPrikryandCohen,Non-GalvinFil}. In order to achieve this, let us choose our independent family a bit differently. The idea is to use Y. Hayut's construction of independent families with special properties from \cite{YairIndep}. Hayut considered these families for the purpose of proving the equivalence between the normal filter extension property and the filter extension property.
\begin{definition}
 A sequence $\l A_i\mid i<\lambda\r$ is called a \textit{normal} $\kappa$-independent family, if it is $\kappa$-independent and for every $\l A_{\alpha_i}\mid i<\kappa\r,\l A_{\beta_i}\mid i<\kappa\r\subseteq \l A_i\mid i<\lambda\r$ are any two disjoint collections, then $\Delta_{i<\kappa}A_{\alpha_i}\setminus A_{\beta_i}$ is a stationary subset of $\kappa$. 
\end{definition}
The following proposition is due to Hayut\footnote{ Hayut uses the more general framework of stationary sets over $P_\kappa(\lambda)$ due to Jech \cite{Jech2010stationary}.}:
\begin{proposition}\label{normalfamily}
 Let $\kappa$ be a regular cardinal. If $\diamondsuit(\kappa)$ holds then there is a normal $\kappa$-independent family of length $2^\kappa$.
\end{proposition}
\begin{proof}
Let $\l X_i\mid i<\kappa\r$ be $\diamondsuit(\kappa)$-sequence. Let us start by multiplying it, by fixing a bijection $\phi:\kappa\times\kappa\times\{0,1\}\leftrightarrow \kappa$. We fix a club $C^*$ of all $\alpha<\kappa$ such that $\phi\restriction \alpha\times\alpha\times\{0,1\}:\alpha\times\alpha\times\{0,1\}\leftrightarrow \alpha$. Consider $\phi^{-1}[X_\alpha]$, for $\alpha\in C^*$, we can identify this with a sequence: $\l \l Y_i^{(\alpha)},Z_i^{(\alpha)}\r\mid i<\alpha\r$.

Take any sequence $\mathcal{A}=\l \l A_i,B_i\r\mid i<\kappa\r$. Then $\mathcal{A}\subseteq \kappa\times \kappa\times\{0,1\}$, and we translate $\phi[\mathcal{A}]=B$.  Then the set $\{\alpha\in C^*\mid B\cap \alpha= X_\alpha\}$ is stationary and for each such $\alpha$, $$B\cap \alpha=\phi[\mathcal{A}\cap(\alpha\times\alpha\times\{0,1\})]=\phi[\l\l A_i\cap\alpha,B_i\cap\alpha\mid i<\alpha\r]$$ 
Hence for each $i<\alpha$, $Y_i^{(\alpha)}=A_i\cap \alpha$ and $Z_i^{(\alpha)}=B_i\cap \alpha$. 

Now suppose that $Y\subseteq \kappa$, define
$$R_Y:=\{\alpha\in C^*\mid \forall i,j<\alpha.Y^{(\alpha)}_i\neq Z^{(\alpha)}_j\text{ and }\exists i<\alpha.Y\cap\alpha=Y^{(\alpha)}_i\}.$$
\begin{claim}
$\l R_Y\mid Y\subseteq\kappa\r$ is a normal $\kappa$-independent family of size $2^\kappa$.
\end{claim}
\begin{proof}
First, if $Y\neq Y'$  let $\nu\in Y\Delta Y'$. Consider the sequence $$\mathcal{Y}=\l \l Y,Y'\r\mid i<\kappa\r.$$ Find $\alpha>\nu$ which guesses $\l \l Y\cap \alpha,Y'\cap\alpha\r\mid i<\alpha\r$. In particular, for each $i,j<\alpha$, $Y^{(\alpha)}_i=Y\cap\alpha\neq Y'\cap \alpha= Z^{(\alpha)}_j$ and therefore $\alpha\in R_Y\Delta R_{Y'}$.

Let $\l R_{Y_i}\mid i<\lambda\r,\l R_{Z_j}\mid j<\lambda'\r$ where $\lambda,\lambda'<\kappa$. For each $i<\lambda,j<\lambda'$ find $\alpha_{i,j}<\kappa$ such that $Y_i\cap \alpha_{i,j}\neq Z_j\cap\alpha_{i,j}$ and let $\alpha^*=\sup_{i,j}\alpha_{i,j}<\kappa$. Consider the sequence $\l Y_i,Z_i\mid i<\kappa\r$ such that if $i$ is not define take $Y_i=Y_0$ or $Z_i=Z_0$. Find $\alpha$ which guesses the sequence $\l\l Y_i\cap\alpha,Z_i\cap\alpha\r\mid i<\alpha\r$ such that $\alpha>\lambda,\lambda',\alpha^*$. Then such $\alpha$ belongs to $\bigcap_{i<\lambda}R_{Y_i}\cap\bigcap_{j<\lambda'}(R_{Z_j})^c$.

Finally, for normality, suppose that $\l Y_i\mid i<\kappa\r$ and $\l Z_i\mid i<\kappa\r$ are two disjoint sequences. Then the set $S$ of all $\alpha\in C^*$ such that the sequence $\l \l Y_i\cap\alpha,Z_i\cap\alpha\r \mid i<\alpha\r$ is guessed is stationary. Also we let $C_0=\{\alpha<\kappa\mid \forall i,j<\alpha, Y_i\cap \alpha\neq Z_j\cap\alpha\}$. Note that $C_0$ is a club at $\kappa$. If $\alpha\in S\cap C_0$, then for every $i<\alpha$, $Y^{(\alpha)}_i=Y_i\cap\alpha\neq Z_i\cap \alpha$ and therefore $\alpha\in R_{Y_i}\setminus R_{Z_i}$. Hence $S\cap C_0\subseteq \Delta_{i<\alpha}R_{Y_i}\setminus R_{Z_i}$. 
\end{proof}
\end{proof}
Note that in a normal $\kappa$-independent family, the intersection of any less than $\kappa$ many of the sets and less than $\kappa$ many of the complement is stationary.
\begin{corollary}\label{SupercomClub}
 If $\kappa$ is $2^\kappa$-supercompact then there is a $\kappa$-complete ultrafilter which extends the club filter and $\neg Gal(U,\kappa,\kappa^+)$.
\end{corollary}
\begin{proof}
First note that since $\kappa$ is $2^\kappa$-supercompact, then $\diamondsuit(\kappa)$ holds. Apply Proposition \ref{normalfamily} to obtain a normal $\kappa$-independent family $\l A_i\mid i<2^\kappa\r$. Next, we use the same construction as in Theorem \ref{SupercompactGalvin} starting with the family $\l A_i\mid i<2^{\kappa}\r$. Note by $2^\kappa$-closure of $M$, $j[Cub_\kappa]$ is a set of less than $j(\kappa)$ many clubs at $j(\kappa)$ and hence $\cap j[Cub_\kappa]$ is a club at $j(\kappa)$. Applying the normality of the family we can find now $$\kappa\leq\delta\in (\bigcap j[Cub_\kappa])\cap (\bigcap_{\alpha<\kappa^+} j(A_{\alpha}))\cap(\bigcap_{\beta\in k[i(\kappa^+)]\setminus j[\kappa^+]}(A'_\beta)^c) $$ By the previous corollary, $(\bigcap_{\alpha<\kappa^+} j(A_{\alpha}))\cap(\bigcap_{\beta\in k[i(\kappa^+)]\setminus j[\kappa^+]}(A'_\beta)^c)$ is stationary in $j(\kappa)$ and therefore such $\delta$ exists. Deriving the same ultrafilter $W=\{X\subseteq \kappa\mid \delta\in j(X)\}$, it is clear that $W$ is a non-Galvin $\kappa$-complete ultrafilter over $\kappa$ which extends the club filter.
\end{proof}
This type of independent family can further be used to answer another question  \cite[Question 3.4]{Non-GalvinFil} about constructing a $\kappa^+$-complete filter $\mathcal{F}$ which extends the club filter on $\kappa^+$ and $\neg Gal(\mathcal{F},\kappa^+,\kappa^{++})$. Here we will prove that there is such a filter if we just assume $\diamondsuit(\kappa)$ and in particular, $L$ is a model where there is always such a filter. This demonstrates, without large cardinals, why Galvin's theorem (and in fact the generalization of it which was considered here) is sharp in the sense that we cannot drop the combinatorial assumption about the filter from Galvin's theorem.
\begin{theorem}\label{nongalfil}
 Suppose that $\kappa$ is a regular cardinal such that $\diamondsuit(\kappa)$ holds, then there is a $\kappa$-complete filter on $\kappa$, extending the club filter, which fails to satisfy the Galvin property.
\end{theorem}
\begin{proof}
    Suppose that $\l A_i\mid i<\kappa^+\r$ is a normal-independent family of $\kappa^+$-many subsets of $\kappa$. Let us define $\mathcal{F}$ to be the minimal $\kappa$-complete filter which extends the family $Cub_\kappa\cup\{A_i\mid i<\kappa^+\}$. Namely, we define $$\mathcal{F}=\{X\subseteq \kappa\mid \exists C\in Cub_\kappa \ \exists I\in[\kappa^+]^{<\kappa}, \ C\cap(\bigcap_{i\in I}A_i)\subseteq X\}$$ This is indeed a $\kappa$-complete filter since the intersection of fewer than $\kappa$ many of the sets $A_i$ is stationary; therefore, this family has the $<\kappa$-intersection property. This guarantees that it generates a $\kappa$-complete filter which we denote by $\mathcal{F}$. Clearly, $\mathcal{F}$ extends the club filter. Let us prove that the family $\l A_i\mid i<\kappa^+\r$ witnesses the failure of the Galvin property. Suppose otherwise, then there is $I\in [\kappa^+]^{\kappa}$ such that $\bigcap_{i\in I}A_i\in \mathcal{F}$. It follows by the definition of $\mathcal{F}$, that there is $J\in [\kappa^+]^{<\kappa}$ and a club $C$ such that $C\cap(\bigcap_{j\in J}A_j)\subseteq \bigcap_{i\in I}A_i$. Pick any $i^*\in I\setminus J$. Such an index exists since $|I|=\kappa$ and $|J|<\kappa$. By normal independence, $(\bigcap_{j\in J}A_j)\cap A^c_{i^*}$ is stationary, and therefore there is some $\nu\in C\cap (\bigcap_{j\in J}A_j)$ such that $\nu\notin A_{i^*}$. But this is absurd since this would mean that $\nu\notin \bigcap_{i\in I}A_i$.
\end{proof}
It follows now, unlike the situation with non-Galvin ultrafilters, that the existence of non-Galvin filters fits well in canonical inner models.
\begin{corollary}\label{Lsituation}
 Assume $V=L$, then every regular cardinal $\kappa$ admits a $\kappa$-complete filter extending the club filter and fails to satisfy the Galvin property.
\end{corollary}
\begin{remark}
     Note that the filter we constructed in Theorem \ref{nongalfil}, extends to a normal $\kappa$-complete filter $\mathcal{F}^*$, namely, 
$$\mathcal{F}^*:=\{X\subseteq \kappa\mid \exists \l\alpha_i\mid i<\kappa\r\in [\kappa^+]^{\kappa}, \ \Delta_{i<\kappa}A_{\alpha_i}\subseteq X\}.$$
As mentioned in the introduction, Galvin's theorem now applies (indeed $\diamondsuit(\kappa)$ implies $\kappa^{<\kappa}=\kappa$) to the normal filter $\mathcal{F}^*$ to conclude it is again Galvin. This is quite an interesting situation as the known methods which ``correct" a non-Galvin ultrafilter and makes it Galvin (see \cite{bgp}) are quite brutal in the sense that we either collapse the generator of an ultrafilter so it extends to an ultrafilter which is isomorphic to a normal one or we diagonalize the filter and make it eventually a $p$-point. 
\end{remark}
A closely related question to the existence of a $\kappa$-complete filter $\mathcal{F}$ extending the club filter such that $Gal(\mathcal{F},\kappa,\kappa^+)$ fails is the question about the $\kappa^+$-saturation of the club filter over $\kappa$ (or the dual non-stationary ideal) which is the assertion that that from every $\kappa^+$-many \textit{stationary} sets there are always two for which the intersection is again stationary. A classical result of Shelah establishes that for any $\kappa>\omega_1$, the non-stationary ideal on $\kappa^+$ is not saturated and this result was later generalized to include inaccessible cardinals by Gitik and Shelah. In light of these results, it would not be too bold to conjecture that it is ZFC-provable that every regular cardinal admits a non-Galvin, $\kappa$-complete filter which extends the club filter.

Let us conclude this paper with one more piece of evidence, which is given below in our last proposition. The proposition also demonstrates that the diamond principle is not necessary for the existence of such a filter and even GCH is not needed. 
\begin{proposition}
    Let $\l A_i\mid i<\lambda\r$ be a sequence of $\lambda$-many mutually generic Cohen subsets of $\kappa$ over $V$, then for each $I,J\in[\lambda]^{<\kappa}$, such that $I\cap J=\emptyset$, $(\bigcap_{i\in I}A_i)\cap(\bigcap_{j\in J}A^c_j)$ is stationary. In particular, if $\mathcal{F}$ is the $\kappa$-complete filter generated by $Cub_\kappa\cup\{A_i\mid i<\lambda\}$, then the sequence $\l A_i\mid i<\lambda\r$ witnesses that $\neg Gal(\mathcal{F},\kappa,\lambda)$. 
\end{proposition}
\begin{remark}
    The idea of using Cohen subsets to generate a non-Galvin filter was already noticed in \cite{Non-GalvinFil} but missed the part about extending the club filter.
\end{remark}
\begin{proof}
    We denote by $Add(\kappa,\lambda):=\{f:\lambda\times\kappa\rightarrow 2\mid |f|<\kappa\}$ ordered as usual by inclusion. Note that the value $f(\l i,\alpha\r)=0,1$ determined weather $\alpha\in A_i$ or $\alpha\notin A_i$ respectively. The proof that the Cohen sets form a $\kappa$-independent family is well known and that the $\kappa$-complete filter generated by the Cohens is non-Galvin can be found in \cite[Lemma 3.1]{Non-GalvinFil}. Now given $I,J$ as above, $(\bigcap_{i\in I}A_i)\cap(\bigcap_{j\in J}A^c_j)$ is stationary in $V[\l A_i\mid i<\lambda\r]$. Indeed, if $\lusim{C}$ is a $Add(\kappa,\lambda)$-name for a club at $\kappa$, and given any condition $p\in Add(\kappa,\lambda)$, construct an increasing sequence of conditions $p_n$ and an increasing sequence of ordinals $\beta_n$  such that $\dom(p_n)\subseteq \lambda\times \beta_n$ and $p_{n+1}\Vdash \beta_n\in \lusim{C}$. Then let $p_\omega=\bigcup_{n<\omega}p_n\in Add(\kappa,\lambda)$ and $\beta_\omega=\sup_{n<\omega}\beta_n$, then $p_\omega\Vdash \beta_\omega\in \lusim{C}$ and $\dom(p_\omega)\subseteq \lambda\times \beta_\omega$. Finally, extend $p_\omega$ to $p'$ by setting $p'(\l i,\beta_\omega\r)=1$ for each $i\in I$ and $p'(\l j,\beta_\omega\r)=0$ for each $j\in J$. By density we can pick such $p'\in G$ and this condition forces $\beta_\omega$ to be in $C\cap \bigcap_{i\in I}A_i\cap\bigcap A_j^c$, as wanted.
\end{proof}
\section{Open questions and acknowledgments}
In this section we collect the open problems which appeared in the previous chapters and add a few more:
\begin{question}
Does $\mathrm{UA}$ imply that every $\kappa$-complete ultrafilter over $\kappa$ satisfies that Galvin property?
\end{question}
\begin{question}
Is there a Galvin ultrafilter which is not Rudin-Keisler equivalent to an ultrafilter of the form appearing in corollary \ref{FormofUlt}?
\end{question}
The next question seeks smaller large cardinals than a supercompact for which we can prove the existence of a non-Galvin ultrafilter.
\begin{question}
    For which large cardinal notions ``$\phi(x)$" the following holds: $\phi(\kappa)$ implies the existence of a non-Galvin $\kappa$-complete ultrafilter over $\kappa$ such that $Cub_\kappa\subseteq U$?  
\end{question}
In this paper we proved that $\phi(x)=``x$ is supercompact" 
 satisfies the above and if $\phi(x)$ is below a superstrong then it is consistent that there is a model with no non-Galvin $\kappa$-complete ultrafilters over $\kappa$. We conjecture that strongly compact cardinals are sufficient and that the filter extension property is the crucial ingredient for the construction of such ultrafilters.
 \begin{question}
     Is it consistent to have a model where there is a regular cardinal $\kappa$ with no non-Galvin filters?
 \end{question}
 Of course, by our results, this should be a model where $\diamondsuit(\kappa)$ fails.
 \subsection*{Added in proof} Theorem \ref{Theorem: general n-fold of kappa p-points} was recently proven differently (and in fact generalized) by N. Dobrinen and the author using the notion of basically generated ultrafilters in \cite{TomNatasha}. Question $5.1$ was essentially answered negatively in \cite{TomGabe} by Goldberg and the author, more precisely, it is shown that under $\mathrm{UA}$ and ``every irreducible is Dodd-sound", an ultrafilter is Galvin if and only if it is an iterated sum of $p$-points. Question 5.2 was answered positively by Gitik \cite{GitikGal}. Question 5.3 was partially answered by Dobrinen and the author in \cite{TomNatasha} where it is shown that if $\kappa$ is a $\kappa$-compact cardinal (i.e. every $\kappa$-complete filter over $\kappa$ can be extended to a $\kappa$-complete ultrafilter over $\kappa$) there is a non-Galvin ultrafilter over $\kappa$ which extends the club filter, the proof there relies heavily on the techniques of this paper. This is further improved in \cite{TomGabe} to the new notion of \textit{non-Galvin cardinals}.
 \subsection*{Acknowledgments} We would like to thank Moti Gitik for many helpful ideas. Also, we thank Gabriel Goldberg for pointing out the characterization of ultrafilters in the canonical inner models up to a superstrong cardinal and for countless insightful discussions. Finally, we would like to thank the referee of this paper who has contributed enormously to improve previous versions of this paper, and for their clever and to-the-point remarks.
 \bibliographystyle{amsplain}
\bibliography{ref}
\end{document}